\documentclass{article}%
\usepackage{graphicx}
\usepackage{amsmath}
\usepackage{amsfonts}
\usepackage{amssymb}%
\providecommand{\U}[1]{\protect\rule{.1in}{.1in}}
\providecommand{\U}[1]{\protect\rule{.1in}{.1in}}
\providecommand{\U}[1]{\protect\rule{.1in}{.1in}}
\providecommand{\U}[1]{\protect\rule{.1in}{.1in}}
\providecommand{\U}[1]{\protect\rule{.1in}{.1in}}
\providecommand{\U}[1]{\protect\rule{.1in}{.1in}}
\providecommand{\U}[1]{\protect\rule{.1in}{.1in}}
\providecommand{\U}[1]{\protect\rule{.1in}{.1in}}
\providecommand{\U}[1]{\protect\rule{.1in}{.1in}}
\providecommand{\U}[1]{\protect\rule{.1in}{.1in}}
\providecommand{\U}[1]{\protect\rule{.1in}{.1in}}
\providecommand{\U}[1]{\protect\rule{.1in}{.1in}}
\providecommand{\U}[1]{\protect\rule{.1in}{.1in}}
\providecommand{\U}[1]{\protect\rule{.1in}{.1in}}
\providecommand{\U}[1]{\protect\rule{.1in}{.1in}}
\providecommand{\U}[1]{\protect\rule{.1in}{.1in}}
\providecommand{\U}[1]{\protect\rule{.1in}{.1in}}
\providecommand{\U}[1]{\protect\rule{.1in}{.1in}}
\providecommand{\U}[1]{\protect\rule{.1in}{.1in}}
\providecommand{\U}[1]{\protect\rule{.1in}{.1in}}
\providecommand{\U}[1]{\protect\rule{.1in}{.1in}}
\providecommand{\U}[1]{\protect\rule{.1in}{.1in}}
\providecommand{\U}[1]{\protect\rule{.1in}{.1in}}
\providecommand{\U}[1]{\protect\rule{.1in}{.1in}}
\providecommand{\U}[1]{\protect\rule{.1in}{.1in}}
\providecommand{\U}[1]{\protect\rule{.1in}{.1in}}
\providecommand{\U}[1]{\protect\rule{.1in}{.1in}}
\providecommand{\U}[1]{\protect\rule{.1in}{.1in}}
\providecommand{\U}[1]{\protect\rule{.1in}{.1in}}

\newtheorem{theorem}{Theorem}
{}

\newtheorem{definition}{Definition}

\newtheorem{lemma}{Lemma}
{}

\newtheorem{remark}{Remark}

\newenvironment{proof}[1][Proof]{\textbf{#1.} }{\ \rule{0.5em}{0.5em}}

\begin{document}

\title{Spectrality of the Dirac Operator with Complex-Valued Periodic Coefficients}
\author{O. A. Veliev\\{\small \ Dogus University, }\\{\small Esenkent 34755, \ Istanbul, Turkey.}\\\ {\small e-mail: oveliev@dogus.edu.tr}}
\date{}
\maketitle

\begin{abstract}
In this paper, we study the spectrality of the non-self-adjoint Dirac operator
$L(Q)$ with a complex-valued periodic $2\times2$ matrix potential $Q.$ We
establish a condition on the off-diagonal elements of the matrix $Q$ under
which $L(Q)$ is an asymptotically spectral operator. Moreover, we derive a
condition on $Q$ that ensures the spectrality of this operator. Finally, we
consider the spectral expansion in these cases.



\end{abstract}

\section{ Introduction and Preliminary Facts}

In this paper, we consider the one-dimensional Dirac operator $L(Q)$ generated
in the space $L_{2}^{2}(-\infty,\infty)$ of $2$-coordinate complex-valued
vector functions $\mathbf{y}(x)\mathbf{=}\left(
\begin{array}
[c]{c}%
y_{1}(x)\\
y_{2}(x)
\end{array}
\right)  $ by the differential expression
\begin{equation}
l(\mathbf{y,}Q)=J\mathbf{y}^{^{\prime}}(x)+Q\mathbf{y}(x), \tag{1}%
\end{equation}
where $J=\left(
\begin{array}
[c]{cc}%
0 & 1\\
-1 & 0
\end{array}
\right)  ,$ $Q(x)=\left(
\begin{array}
[c]{cc}%
a_{1}(x) & a_{2}(x)\\
a_{3}(x) & a_{4}(x)
\end{array}
\right)  $ and $a_{k}$ for $k=1,2,3,4$ are $\pi$-periodic, complex-valued
functions of bounded variation on $[0,\pi]$. We prove that if
\begin{equation}
\operatorname{Re}%
{\textstyle\int\nolimits_{0}^{\pi}}
a_{3}(x)-a_{2}(x)dx\neq0, \tag{2}%
\end{equation}
then $L(Q)$ is an asymptotically spectral operator. Moreover, we derive a
condition on $Q$ that ensures the spectrality of this operator. Finally, we
discuss the spectral expansion in these cases.

It is well-known [7, 13, 14, 18] that the spectrum $\sigma(L(Q))$ of the
operator $L(Q)$ is the union of the spectra $\sigma(L_{t}(Q))$ of the
operators $L_{t}(Q)$ for $t\in(-1,1]$ generated in $L_{2}^{2}[0,\pi]$ by (1)
and the boundary condition
\begin{equation}
\mathbf{y}(\pi)=e^{i\pi t}\mathbf{y}(0). \tag{3}%
\end{equation}
According to the definition given in [14] (see Definition 3.2 of [14] ), we
use the following definition of the spectrality and asymptotic spectrality

\begin{definition}
Let $\gamma(t)$ be a closed contour lying in the resolvent set $\rho\left(
L_{t}(Q)\right)  $ of $L_{t}(Q).$ Define $e(\gamma(t))$ by
\[
e(\gamma(t))=\int_{\gamma(t)}\left(  L_{t}(Q)-\lambda I\right)  ^{-1}%
d\lambda\text{ }.
\]
The operator $L$ is said to be a spectral operator if
\begin{equation}
\sup_{t\in(-1,1]}\left(  \sup_{\gamma(t)\subset\rho\left(  L_{t}\right)
}\left\Vert e(\gamma(t))\right\Vert \right)  <\infty, \tag{4}%
\end{equation}
where $\sup$ is taken for all $\gamma(t)\subset\rho\left(  L_{t}\right)  .$
The operator $L$ is said to be an asymptotically spectral operator if there
exists $M$ such that (4) holds for all $\gamma(t)\subset\left(  \{\lambda
\in\mathbb{C}:\mid\lambda\mid>M\}\cap\rho\left(  L_{t}\right)  \right)  .$
\end{definition}

Note that in papers [19--23], we considered the spectrality for operator
generated by differential expression of order $m\geq2.$ Those investigations
does not imply the spectrality for the Dirac operator generated by first-order
differential expression (1). It is important to note that, the one dimensional
Schrodinger operator $T(q),$ generated in the space $L_{2}(-\infty,\infty)$ by
the differential expression $-y^{\prime\prime}+qy$ with complex -valued
potential $q$ is, in general, not a spectral operator. Gesztezy and Tkachenko
[8] proved two versions of a criterion for the operator $T(q)$ to be a
spectral operator, in sense of Dunford [5], one analytic and one geometric.
The analytic version was stated in terms of the solutions of the Hill
equation. The geometric version of the criterion uses algebraic and
geometric\ properties of the spectra of the periodic/antiperiodic and
Dirichlet boundary value problems. The problem of explicitly describing for
which potentials $q$ the Schrodinger operators $T(q)$ are spectral operators
has remained open for about 65 years. In [19] (see also [22, Sect. 2.7]), I
found explicit conditions on the potential $q$ such that $T(q)$ is an
asymptotically spectral operator. However, the following well-known examples
demonstrate that the spectrality of $T(q)$ is a very rare phenomenon. Gasimov
[6] proved that the operator $T(q)$ with the potential
\[
q(x)=%
{\textstyle\sum\limits_{n=1}^{\infty}}
q_{n}e^{i2\pi x}%
\]
has, in general, infinitely many spectral singularities. Recall that the
spectral singularity of $T(q)$ is a point of its spectrum $\sigma(T(q))$ in
neighborhood on which the projections of $T(q)$ are not uniformly bounded.
Therefore, existence a spectral singularity does not allow $T(q)$ to be a
spectral operator. In the case $q(x)=ae^{i2\pi x}$ and $a\neq0.$ The numbers
$\left(  \pi n\right)  ^{2}$ for $n\neq0$ are the spectral singularities for
all $n\in\mathbb{Z}$ (see [22, Sect. 3.3]).

The other demonstration is the Mathieu--Schr\"{o}dinger operators $H(a,b)$
generated in $L_{2}(-\infty,\infty)$ by differential expression
\begin{equation}
-y^{\prime\prime}+(ae^{i2\pi x}+be^{-i2\pi x})y. \tag{6}%
\end{equation}
Djakov and Mityagin [1] proved that the root functions of the operators
$H_{0}(a,b)$ or $H_{\pi}(a,b)$ generated in $L_{2}[0,1]$ by periodic or
antiperiodic boundary conditions form a Riesz basis if and only if $\mid
a\mid=\mid b\mid$ (see Theorem 7 in [1]). Moreover, in [1] (see page 539) it
was noted that Theorem 7 in [1] with Remark 8.10 in [8] imply that $\ H(a,b)$
is not a spectral operator, if $|a|\neq\left\vert b\right\vert .$ For another
proof of this statement, see Section 4.3 in [22] (Proposition 4.3.1). It is
important to emphasize that the spectrality of $H_{0}(a,b)$ and $H_{\pi}%
(a,b)$, and even the spectrality of the operators $H_{t}(a,b)$ generated in
$L_{2}[0,1]$ by $t$-periodic boundary conditions, for all $t\in(-\pi,\pi]$,
does not imply the spectrality of $H(a,b)$. In Remark 4.3.5 of [22], I noted
that a detailed investigation of $H_{0}(a,b)$ and $H_{\pi}(a,b)$ cannot
provide a sufficient condition for the spectrality of $\ H(a,b)$. This is
because multiple eigenvalues of $H_{t}(a,b)$ for $t\neq0,\pi$ become spectral
singularities of $H(a,b)$, which in turn implies that $H(a,b)$ is not a
spectral operator. Such multiple eigenvalues may indeed arise for $t\neq0,\pi$
(see Theorem 4.3.1 in [22]). Thus, the condition $\mid a\mid=\mid b\mid$ is
not sufficient for the spectrality of $H(a,b)$. In [21], Theorem 1 and in [22]
Theorem 4.1.1, I proved that $H(a,b)$ is an asymptotically spectral operator
if and only if $\mid a\mid=\mid b\mid$ and
\begin{equation}
\text{ }\inf_{q,p\in\mathbb{N}}\{\mid q\alpha-(2p-1)\mid\}\neq0, \tag{7}%
\end{equation}
where $\alpha=\pi^{-1}\arg(ab)$, $\mathbb{N=}\left\{  1,2,...,\right\}  .$ For
example, if $\alpha=\frac{m}{q},$ where $m$ is an odd integer and $m,q$ are
coprime, then $H(a,b)$ is not even an asymptotically spectral operator (see
[22], Theorem 4.3.4). In other word, even when $\mid a\mid=\mid b\mid$ if
condition (7) fails, then $H(a,b)$ is not a spectral operator. Moreover,
Theorem 1 of [21] implies that if $ab\in\mathbb{R}$, then $T(q)$ is a spectral
operator if and only if it is self adjoint (see Corollary 1 of [21]). Thus, in
this case, there are no spectral operators that are not self-adjoint. These
examples show that the spectrality of $T(q)$ is a very rare phenomenon and
does not depend on the smoothness or smallness of $q.$ We encounter the same
situation for the Dirac operator \ $L(Q)$ when $q_{2}=q_{3}.$ Therefore,
condition (2) is essential and helps us to consider the spectrality of $L(Q).$
Finally, note that in [23], to demonstrate that the theory of spectral
operators is not applicable to the non-self-adjoint Schr\"{o}dinger operator
$T(q)$, I discussed only examples for which there exist necessary and
sufficient conditions on $q$ for the spectrality of $T(q).$

In next section (Section 2) we investigate asymptotic formulas for the Bloch
eigenvalues and Bloch functions. In Section 3, we investigate the asymptotic
spectrality and spectrality of $L(Q)$ and discuss the corresponding spectral
expansion. Throughout the paper, $m_{1},m_{2},...$ denote positive constants
that do not depend on $t,$ $\lambda$ and $x.$ They are used in the sense that,
for some inequality, there exists a constant $m_{i}$ such that the inequality holds.

\section{Asymptotic Formulas for Bloch eigenvalues and Bloch functions}

Let $\mathbf{c}(x,\lambda)=\left(
\begin{array}
[c]{c}%
c_{1}(x,\lambda)\\
c_{2}(x,\lambda)
\end{array}
\right)  $ and $\mathbf{s}(x,\lambda)=\left(
\begin{array}
[c]{c}%
s_{1}(x,\lambda)\\
s_{2}(x,\lambda)
\end{array}
\right)  $ be the solutions of the equation%
\begin{equation}
l(\mathbf{y,}Q)=\lambda\mathbf{y} \tag{8}%
\end{equation}
satisfying the following initial conditions%

\begin{equation}
c_{1}(0,\lambda)=s_{2}(0,\lambda)=1,\text{ }s_{1}(0,\lambda)=c_{2}%
(0,\lambda)=0. \tag{9}%
\end{equation}
Substituting the general solution $x_{1}\mathbf{c}(x,\lambda)+x_{2}$
$\mathbf{s}(x,\lambda)$ into boundary condition (3) and using (9), we obtain
the following system of equations
\begin{equation}
\left\{
\begin{array}
[c]{c}%
x_{1}(c_{1}(\pi,\lambda)-e^{i\pi t})+x_{2}s_{1}(\pi,\lambda)=0\\
x_{1}c_{2}(\pi,\lambda)+x_{2}(s_{2}(\pi,\lambda)-e^{i\pi t})=0
\end{array}
\right.  . \tag{10}%
\end{equation}
A number $\lambda$ is an eigenvalue of $L_{t}(Q)$ if and only if it is a root
of the characteristic equation
\[
\Delta(\lambda,t)=\left\vert
\begin{array}
[c]{cc}%
c_{1}(\pi,\lambda)-e^{i\pi t} & s_{1}(\pi,\lambda)\\
c_{2}(\pi,\lambda) & s_{2}(\pi,\lambda)-e^{i\pi t}%
\end{array}
\right\vert =e^{2i\pi t}-e^{i\pi t}F(\lambda)+W(\pi,\lambda)=0,
\]
where $F(\lambda):=c_{1}(\pi,\lambda)+s_{2}(\pi,\lambda)$ and $W(\pi
,\lambda)=c_{1}(\pi,\lambda)s_{2}(\pi,\lambda)-c_{2}(\pi,\lambda)s_{1}%
(\pi,\lambda).$ Thus, the eigenvalues $\lambda(t)$ of $L_{t}(Q)$ are the roots
of the equation%
\begin{equation}
F(\lambda)=e^{i\pi t}+e^{-i\pi t}W(\pi,\lambda). \tag{11}%
\end{equation}
If $\lambda(t)$ is a simple eigenvalue and $s_{1}(\pi,\lambda(t))\neq0,$ then
the corresponding eigenfunction $\Phi_{t}(x)$ can be written in the form
\begin{equation}
\Phi_{t}(x)=\mathbf{c}(x,\lambda(t))+\frac{e^{i\pi t}-c_{1}(\pi,\lambda
(t))}{s_{1}(\pi,\lambda(t))}\mathbf{s}(x,\lambda(t)). \tag{12}%
\end{equation}

It is well-known (see [25]) that if $f$ is a function of bounded variation on
$[0,\pi],$ then
\begin{equation}
\int_{0}^{x}f(t)\cos\lambda tdt=O\left(  \frac{1}{\lambda}\right)  ,\text{
}\int_{0}^{x}f(t)\sin\lambda tdt=O\left(  \frac{1}{\lambda}\right)  , \tag{13}%
\end{equation}
as $\lambda\rightarrow\infty$. Moreover, $O\left(  \frac{1}{\lambda}\right)  $
is independent of $x\in\lbrack0,\pi];$ that is, these estimates hold uniformly
with respect to $x\in\lbrack0,\pi].$ Therefore, by applying Theorems 1 and 2
of [15] or Theorems 2.5 and 2.6 of [16], we obtain the following asymptotic
formulas, which are uniform with respect to $x\in\lbrack0,\pi]$:
\begin{equation}
s_{1}(x,\lambda)=-e^{a(x)}\sin\lambda x+O\left(  \frac{1}{\lambda}\right)
,\text{ }s_{2}(x,\lambda)=e^{a(x)}\cos\lambda x+O\left(  \frac{1}{\lambda
}\right)  \tag{14}%
\end{equation}
and
\begin{equation}
c_{1}(x,\lambda)=e^{a(x)}\cos\lambda x+O\left(  \frac{1}{\lambda}\right)
,\text{ }c_{2}(x,\lambda)=e^{a(x)}\sin\lambda x+O\left(  \frac{1}{\lambda
}\right)  , \tag{15}%
\end{equation}
where $a(x)=$ $\frac{1}{2}%
{\textstyle\int\nolimits_{0}^{x}}
a_{3}(\xi)-a_{2}(\xi)d\xi.$ Hence, we have
\[
F(\lambda)=2e^{\pi b}\cos\pi\lambda+O\left(  \frac{1}{\lambda}\right)  ,\text{
}W(\pi,\lambda)=e^{2\pi b}+O\left(  \frac{1}{\lambda}\right)  ,
\]
where $b=\frac{1}{2\pi}%
{\textstyle\int\nolimits_{0}^{\pi}}
a_{3}(x)-a_{2}(x)dx=\frac{1}{\pi}a(\pi)$. \ Thus, (11) can be written in the
form
\[
2e^{\pi b}\cos\pi\lambda=e^{i\pi t}+e^{-i\pi t}e^{2\pi b}+O\left(  \frac
{1}{\lambda}\right)  .
\]
Dividing this equality by $2e^{\pi b}$ and using well-known Euler's formulas
we obtain
\begin{equation}
\cos\pi\lambda=\cos\pi(2n\pm t\pm ib)+g(\lambda),\text{ }\left\vert
g(\lambda)\right\vert <\frac{m_{1}}{\left\vert \lambda\right\vert } \tag{16}%
\end{equation}
as $\lambda\rightarrow\infty$ for some $m_{1}>0.$ Now, using (16),
Rouch\'{e}'s theorem and (12), in the standard way, we obtain the following
asymptotic formulas.

\begin{theorem}
\textit{\ If} (2) holds, \textit{then: }

$(a)$ \textit{The eigenvalues of }$L_{t}(Q),$ for $t\in(-1,1),$%
\textit{\ consist of two sequences }$\left\{  \lambda_{n,1}(t);\text{ }%
n\in\mathbb{Z}\right\}  $ and $\left\{  \lambda_{n,2}(t);\text{ }%
n\in\mathbb{Z}\right\}  $\ which satisfy the following asymptotic formulas%
\begin{equation}
\lambda_{n,1}(t)=(2n-t-ib)+O(\frac{1}{n}),\text{ }\lambda_{n,2}%
(t)=(2n+t+ib)+O(\frac{1}{n}). \tag{17}%
\end{equation}
These formulas are uniform with respect to $t\in(-1,1].$ There exists $N$ such
that $\lambda_{n,j}(t)$ are the simple eigenvalues of $L_{t}(Q)$ for all
$\left\vert n\right\vert \geq N,$ $j=1,2,$ and $t\in(-1,1].$

$(b)$ The\textit{ }eigenfunctions $\Phi_{n,1,t}(x)$ and $\Phi_{n,2,t}(x)$ of
$L_{t}(Q),$ defined by (12) and corresponding to the eigenvalue,
$\lambda_{n,1}(t)$ and $\lambda_{n,2}(t)$\ \textit{satisfy} the following
asymptotic formula%
\begin{equation}
\Phi_{n,1,t}(x)=\left(
\begin{array}
[c]{c}%
1\\
i
\end{array}
\right)  e^{a(x)}e^{i(-2n+t+ib)x}+O(\frac{1}{n}),\text{ }\Phi_{n,2,t}%
(x)=\left(
\begin{array}
[c]{c}%
1\\
-i
\end{array}
\right)  e^{a(x)}e^{i(2n+t+ib)x}+O(\frac{1}{n})\text{ }, \tag{18}%
\end{equation}
These formulas are uniform with respect to $t\in(-1,1]$ and $x\in\lbrack
0,\pi].$
\end{theorem}

\begin{proof}
$(a)$ To prove (17), we show that there exists $m_{2}$ such that
\begin{equation}
\left\vert 2\cos\pi\lambda-2\cos\pi(2n\pm t\pm ib)\right\vert >\frac{m_{1}%
}{\left\vert \lambda\right\vert } \tag{19}%
\end{equation}
for all
\[
\lambda=2n\pm t\pm ib+\frac{m_{2}}{n}e^{i\alpha},\text{ }\alpha\in
\lbrack0,2\pi),
\]
where $m_{1}$ is defined in (16). Let us consider the case where $\pm$ is
replaced by $-;$ the other case is similar. Using the difference-to-product
trigonometric formula
\begin{align*}
&  \cos\pi(2n-t-ib)+\frac{m_{2}}{n}e^{i\alpha})-\cos\pi(2n-t-ib)\\
&  =-2\sin\left(  (2n-t-ib)\pi+\frac{m_{2}}{2n}e^{i\alpha}\right)  \sin
\frac{m_{2}}{2n}e^{i\alpha},
\end{align*}
together with (2), and the obvious inequalities
\[
\left\vert \sin\left(  (2n-t-ib)\pi+\frac{m_{2}}{2n}e^{i\alpha}\right)
\right\vert >m_{3}\left\vert \operatorname{Re}b\right\vert ,\text{ }\left\vert
\sin\frac{m_{2}}{2n}e^{i\alpha}\right\vert >\left\vert \frac{m_{2}}%
{3n}\right\vert ,
\]
we obtain (19). Therefore, the first formula in (17) follows from (16) and
Rouch\'{e}'s theorem. In the same way, we obtain the proof of the second
formula in (17).

$(b)$ To prove (18), we use (12), (14) and (15). It follows from (17), (14)
and (15) that%
\[
s_{1}(\pi,\lambda_{n,2}(t))=-e^{\pi b}\sin\pi(t+ib)+O(\frac{1}{n}),\text{
}c_{1}(\pi,\lambda_{n,2}(t))=e^{\pi b}\cos\pi(t+ib)+O(\frac{1}{n}),
\]%
\[
e^{i\pi t}-c_{1}(\pi,\lambda_{n,2}(t))=e^{\pi b}e^{i\pi(t+ib)}-e^{\pi b}%
\cos\pi(t+ib)=ie^{\pi b}\sin\pi(t+ib),
\]%
\[
\mathbf{c}(x,\lambda_{n,2}(t))=\left(
\begin{array}
[c]{c}%
e^{a(x)}\cos(2n+t+ib)x\\
e^{a(x)}\sin(2n+t+ib)x
\end{array}
\right)  +O(\frac{1}{n})
\]
and%
\[
\mathbf{s}(x,\lambda_{n,2}(t))=\left(
\begin{array}
[c]{c}%
-e^{a(x)}\sin(2n+t+ib)x\\
e^{a(x)}\cos(2n+t+ib)x
\end{array}
\right)  +O(\frac{1}{n}).
\]
Using these equalities in (12), we obtain the proof of the second formula of
(18). In the same way we prove the first formula in (18).
\end{proof}

Thus, for large values of $n,$ the eigenvalues $\lambda_{n,1}(t)$ and
$\lambda_{n,2}(t)$ of $L_{t}(Q)$ are asymptotically close, respectively, to
the eigenvalues $(2n-t-ib)$ and $(2n+t+ib)$ of the operator $L_{z}(O),$ where
$z=t+ib$ and $O$ is $2\times2$ zero operator. The corresponding eigenfunctions
of $L_{z}(O)$ are
\[
\left(
\begin{array}
[c]{c}%
1\\
i
\end{array}
\right)  e^{i(-2n+t+ib)x}\text{ }\And\text{ }\left(
\begin{array}
[c]{c}%
1\\
-i
\end{array}
\right)  e^{i(2n+t+ib)x}.
\]
It is clear that if (2) holds, that is, if $\operatorname{Re}b\neq0,$ then all
eigenvalues of $L_{z}(O)$ for $z=t+ib$ are simple and the distance between
neighboring eigenvalues in not less than $\min\left\{  2,2\left\vert
\operatorname{Re}b\right\vert \right\}  $ . Moreover, it can be easily
verified that the boundary condition (2) is strongly regular for
$t\notin\mathbb{Z}$ in the sense of [16]. This verification was carried out in
[24]. Therefore, if (2) holds, then the operator $L_{t}(Q)$ is strongly
regular and its root functions, for all $t\in(-1,1],$ forms a Riesz basis in
$L_{2}^{2}[0,\pi]$ according to [16].

Note that there are many papers on the basis properties of the root functions
of Dirac operators in $L_{2}^{2}[a,b],$ for $-\infty<a<b<\infty,$ under
various boundary conditions (see [2-4, 9-12, 15-17] and the references
therein). Here, we use only the fact that the root functions of the Dirac
operator generated by (1) with strongly regular boundary conditions form a
Riesz basis. Therefore, we do not discuss in detail the works devoted to Dirac
operators on a bounded interval.\ 

Since the matrix
\[
\overline{Q}=\left(
\begin{array}
[c]{cc}%
\overline{a_{1}(x)} & \overline{a_{3}(x)}\\
\overline{a_{2}(x)} & \overline{a_{4}(x)}%
\end{array}
\right)
\]
also satisfies condition (2), by Theorem 1, the eigenfunctions $\Phi
_{n,1,t}^{\ast}(x)$ and $\Phi_{n,2,t}^{\ast}(x)$ of the adjoint operator
$L_{t}(\overline{Q})$ satisfy the asymptotic formulas%
\begin{equation}
\Phi_{n,1,t}^{\ast}(x)=\left(
\begin{array}
[c]{c}%
1\\
i
\end{array}
\right)  e^{-\overline{a(x)}}e^{i(-2n+t-i\overline{b})x}+O(\frac{1}{n}),\text{
}\Phi_{n,2,t}^{\ast}(x)=\left(
\begin{array}
[c]{c}%
1\\
-i
\end{array}
\right)  e^{-\overline{a(x)}}e^{i(2n+t-i\overline{b})x}+O(\frac{1}{n})\text{
}. \tag{20}%
\end{equation}

\begin{remark}
One can easily verify that the numbers $2n-t-ib$ and $2n+t+ib$ are the
eigenvalues of the operator $L_{t}(Q_{b}),$ where $Q_{b}=\left(
\begin{array}
[c]{cc}%
0 & -b\\
b & 0
\end{array}
\right)  .$ The eigenfunctions of $L_{t}(Q_{b})$ corresponding to the
eigenvalues $(2n-t-ib)$ and $(2n+t+ib)$ are
\[
\left(
\begin{array}
[c]{c}%
1\\
i
\end{array}
\right)  e^{i(-2n+t)x}\text{ and }\left(
\begin{array}
[c]{c}%
1\\
-i
\end{array}
\right)  e^{i(2n+t)x}.
\]
Its adjoint operator $\left(  L_{t}(Q_{b})\right)  ^{\ast}$ is $L_{t}%
(\overline{Q_{b}}),$ where $\overline{Q_{b}}=\left(
\begin{array}
[c]{cc}%
0 & \overline{b}\\
-\overline{b} & 0
\end{array}
\right)  .$ The eigenvalues of $L_{t}(\overline{Q_{b}})$ are $(2n-t+i\overline
{b})$ and $(2n+t-i\overline{b}).$ The normalized eigenfunctions corresponding
to these eigenvalues are
\[
\varphi_{n,1,t}(x)=\frac{1}{\sqrt{2\pi}}\left(
\begin{array}
[c]{c}%
1\\
i
\end{array}
\right)  e^{i(-2n+t)x}\text{ and }\varphi_{n,2,t}(x)=\frac{1}{\sqrt{2\pi}%
}\left(
\begin{array}
[c]{c}%
1\\
-i
\end{array}
\right)  e^{i(2n+t)x}.
\]

\end{remark}

\section{Spectrality and Spectral Expansion}

To study spectrality, we use Definition 1; that is, estimate the projections
\[
e(\gamma(t))=\int_{\gamma(t)}\left(  L_{t}(Q)-\lambda I\right)  ^{-1}%
d\lambda,\text{ }%
\]
where $\gamma(t)$ is a closed curve lying in $\{\lambda\in\mathbb{C}%
:\mid\lambda\mid>M\}\cap\rho\left(  L_{t}\right)  $. By Theorem 1, there
exists $M$ such that the closed curve $\gamma(t)$ encloses only the
eigenvalues $\lambda_{n,j}(t)$ with $\mid n\mid\geq N,$ where $N$ is defined
in Theorem 1$(a).$ This mean that $\lambda_{n,j}(t)$ is a simple eigenvalues.
Let $\mathbb{D}(t)$ be a subset of the set
\[
\{(n,j):n\in\mathbb{Z},\mid n\mid\geq N,\text{ }j=1,2\}=\mathbb{Z}%
\times\{1,2\}
\]
such that $\lambda_{n,j}(t)$ lies inside $\gamma(t)$ for $\left(  n,j\right)
\in\mathbb{D}(t).$ Then we have
\begin{equation}
e(\gamma(t))f=\text{ }\sum_{\left(  n,j\right)  \in\mathbb{D}(t)}\frac
{1}{\alpha_{n,j}(t)}(f,\Phi_{n,j,t}^{\ast})\Phi_{n,j,t}, \tag{21}%
\end{equation}
where $\alpha_{n,j}(t)=(\Phi_{n,j,t},\Phi_{n,j,t}^{\ast}).$ Now, using this,
we prove the following.

\begin{theorem}
If (2) holds, then $L(Q)$ is an asymptotically spectral operator.
\end{theorem}

\begin{proof}
By (18) and (20), we have
\begin{equation}
\Phi_{n,j,t}=e^{a(x)-bx}d_{j}e^{i(\left(  -1\right)  ^{j}2n+t)x}+O(\frac{1}%
{n}),\text{ }\Phi_{n,j,t}^{\ast}=e^{-\overline{a(x)}+\overline{b}x}%
d_{j}e^{i(\left(  -1\right)  ^{j}2n+t)x}+O(\frac{1}{n}) \tag{22}%
\end{equation}
and $\alpha_{n,j}(t)=2+O(\frac{1}{n}),$ where $d_{j}=\left(
\begin{array}
[c]{c}%
1\\
\left(  -1\right)  ^{j-1}i
\end{array}
\right)  .$ Using (22) in (21), we obtain
\begin{equation}
e(\gamma(t))f=\frac{1}{2}e^{a(x)-bx}\sum_{\left(  n,j\right)  \in
\mathbb{D}(t)}\left(  e^{-a(x)+bx}f,d_{j}e^{i(\left(  -1\right)  ^{j}%
2n+t)x}+O(\frac{1}{n})\right)  \left(  d_{j}e^{i(\left(  -1\right)
^{j}2n+t)x}+O(\frac{1}{n})\right)  . \tag{23}%
\end{equation}
Now, using (23) and repeating the proof of (2.5) of [20], we conclude that
there exists a positive constant $m_{4}$ and $m_{5}$ independent of $t$,
$\mathbb{D}(t)$ and $f$ such that
\begin{equation}
\left\Vert e(\gamma(t))f\right\Vert ^{2}\leq m_{4}\parallel e^{-a(x)+bx}%
f\parallel^{2}\leq m_{5}\left\Vert f\right\Vert ^{2} \tag{24}%
\end{equation}
for all $f\in L_{2}^{2}(0,\pi)$, $t\in(-1,1]$ and $\gamma(t)\subset
\{\lambda\in\mathbb{C}:\mid\lambda\mid>M\}\cap\rho\left(  L_{t}\right)  .$
Therefore the proof of this theorem follows from Definition 1.
\end{proof}

Using (24) and repeating the proof of (2.18) of [20], we obtain the following result.

\begin{theorem}
If (2) is satisfied, then for each $f$ $\in L_{2}^{m}(-\infty,\infty)$ the
following equality holds
\[
f(x)=\frac{1}{2}\int\limits_{(-1,1]}%
{\displaystyle\sum\limits_{j=1,2;\text{ }\left\vert k\right\vert \leq N}}
a_{k,j}(t)\Psi_{k,j,t}(x)dt+\frac{1}{2}%
{\displaystyle\sum\limits_{j=1,2;\text{ }\left\vert k\right\vert >N_{{}}}}
\int\limits_{(-1,1]}a_{k,j}(t)\Psi_{k,j,t}(x)dt,
\]
where $a_{k,j}(t)=(f,X_{k,j,t}),$ $\Psi_{k,j,t}$ is the normalized
eigenfunction corresponding to $\lambda_{k,j}(t),$ and $\left\{
X_{k,j,t}:k\in\mathbb{Z},\text{ }j=1,2\right\}  $ is the biortogonal system to
$\left\{  \Psi_{k,j,t}:k\in\mathbb{Z},\text{ }j=1,2\right\}  .$ The series in
this formula converges in the norm of $L_{2}^{2}(a,b)$ for every
$a,b\in\mathbb{R}.$
\end{theorem}

Now, we determine a condition on the potential $Q$ under which $L_{t}(Q)$ is a
spectral operator. To this end, we look for a condition on $Q$ that guarantees
the simplicity of the eigenvalues of $L_{t}(Q)$ for all $t\in(-1,1].$ For this
purpose, we consider the operator $L_{t}(Q)$ as a perturbation of the operator
$L_{t}(Q_{b})$ by $Q-Q_{b}.$ The eigenvalues of $L_{t}(Q_{b})$ are simple (see
Remark 1) and lie on the horizontal lines
\begin{equation}
H(\pm b)=\left\{  (x,y)\in\mathbb{R}^{2}:y=\operatorname{Re}b\right\}  \text{
and }H(\pm b)=\left\{  (x,y)\in\mathbb{R}^{2}:y=-\operatorname{Re}b\right\}  .
\tag{25}%
\end{equation}
The distance between neighboring eigenvalues of $L_{t}(Q_{b})$ is either $2$
(if they lie on the same line) or not less than $2\left\vert \operatorname{Re}%
b\right\vert $ (if they lie on the different line). Therefore, if the norm of
$Q-Q_{b}$ is sufficiently small, then the eigenvalues of $L_{t}(Q)$ remain
simple for all $t\in(-1,1].$ To determine this bound, we use the following
equality
\begin{equation}
(\lambda-(2n+\left(  -1\right)  ^{j}(t+ib))(\Psi_{\lambda},\varphi
_{n,j,t})=(\Psi_{\lambda},(Q-Q_{b})\varphi_{n,j,t}), \tag{26}%
\end{equation}
where $\Psi_{\lambda}$ is the normalized eigenfunction of $L_{t}(Q)$
corresponding to the eigenvalue $\lambda$ and $\varphi_{n,j,t}$ is defined in
Remark 1. \ Equality (26) follows from the relation $L_{t}(Q)\Psi_{\lambda}=$
$\lambda\Psi_{\lambda},$ by taking the inner product with $\varphi_{n,j,t}$
and using
\[
\left(  L_{t}(Q_{b})\right)  ^{\ast}\varphi_{n,j,t}=(2n+\left(  -1\right)
^{j}(t-i\overline{b}))\varphi_{n,j,t}.
\]
The following lemmas play a critical role in establishing the simplicity of
the eigenvalues of $L_{t}(Q)$ for all $t\in(-1,1].$

\begin{lemma}
If $\left\vert \operatorname{Re}b\right\vert \geq1$ and
\begin{equation}
\int\limits_{0}^{\pi}\left\vert a_{1}(x)\pm\left(  a_{2}(x)+b\right)
i\right\vert ^{2}+\left\vert a_{3}(x)-b\pm\imath a_{4}(x)\right\vert
^{2}dx\leq\frac{4}{\pi}, \tag{27}%
\end{equation}
then the circles
\[
C(1,2k\pm t\pm ib)=\left\{  \lambda\in\mathbb{C}:\left\vert \lambda-(2k\pm
t\pm ib)\right\vert =1\right\}
\]
for all $k\in\mathbb{Z}$ lie in the resolvent set of $L_{t}(Q)$.
\end{lemma}

\begin{proof}
Assume that there exists $\lambda\in C(1,2k+t+ib)$ which is an eigenvalue of
$L_{t}(Q)$. Then, by (26), we have
\begin{equation}
\left\vert \left(  \Psi_{\lambda},\varphi_{n,j,t}\right)  \right\vert ^{2}%
\leq\frac{\left\Vert (Q-Q_{b})\varphi_{n,j,t}\right\Vert ^{2}}{\left\vert
\lambda-(2n+\left(  -1\right)  ^{j}(t+ib))\right\vert ^{2}}. \tag{28}%
\end{equation}
Using the definitions of $Q,$ $Q_{b}$ and $\varphi_{n,j,t}$ together with
condition (27), we obtain
\[
\left\Vert (Q-Q_{b})\varphi_{n,j,t}\right\Vert ^{2}\leq\frac{2}{\pi^{2}},
\]
for all $n,j$ and $t.$ Therefore, if
\begin{equation}
\sum_{n\in\mathbb{Z}}\frac{1}{\left\vert \lambda-(2n+t+ib)\right\vert ^{2}%
}+\sum_{n\in\mathbb{Z}}\frac{1}{\left\vert \lambda-(2n-t-ib)\right\vert ^{2}%
}<\frac{1}{2}\pi^{2}, \tag{29}%
\end{equation}
then we obtain the following contradiction:
\begin{equation}
1=\left\Vert \Psi_{\lambda}\right\Vert ^{2}=\sum_{j=1,2;n\in\mathbb{Z}%
}\left\vert \left(  \Psi_{\lambda},\varphi_{n,j,t}\right)  \right\vert ^{2}<1,
\tag{30}%
\end{equation}
since $\left\{  \varphi_{n,j,t}:j=1,2;\text{ }n\in\mathbb{Z}\right\}  $ is an
orthonormal basis.

Thus, it remains to prove (29). First, consider the first summation in (29).
By assumption $\lambda=2k+t+ib+x+iy,$ where $x\in\lbrack-1,1]$ and $y=\pm
\sqrt{1-x^{2}}.$ Then
\begin{equation}
\left\vert \lambda-(2n+t+ib)\right\vert ^{2}=\left(  2s+x\right)  ^{2}%
+1-x^{2}=4s^{2}+4sx+1, \tag{31}%
\end{equation}
where $s=k-n$. Therefore, we have
\[
\sum_{n\in\mathbb{Z}}\frac{1}{\left\vert \lambda-(2n+t+ib)\right\vert ^{2}%
}=\sum_{s\in\mathbb{Z}}\frac{1}{4s^{2}+4sx+1}=1+
\]%
\[
\sum_{s\in\mathbb{N}}\frac{1}{4s^{2}+4sx+1}+\sum_{s\in\mathbb{N}}\frac
{1}{4s^{2}-4sx+1}=1+\sum_{s\in\mathbb{N}}\frac{2(4s^{2}+1)}{(4s^{2}%
+1)^{2}-16s^{2}x^{2}}.
\]
Since the function $(4s^{2}+1)^{2}-16s^{2}x^{2}$ on the interval $[-1,1]$
attains its minimum value at the points $x=1$ or $x=-1,$ in the both cases,
using the well-known equality
\[
\sum_{n=1}^{\infty}\frac{1}{(2n-1)^{2}}=\frac{1}{8}\pi^{2},
\]
we obtain
\begin{equation}
\sum_{n\in\mathbb{Z}}\frac{1}{\left\vert \lambda-(2n+t+ib)\right\vert ^{2}%
}\leq1+\sum_{n=1}^{\infty}\frac{1}{(2n+1)^{2}}+\sum_{n=1}^{\infty}\frac
{1}{(2n-1)^{2}}=\frac{1}{4}\pi^{2}. \tag{32}%
\end{equation}

Now, consider the second summation in (29). Since%
\[
\lambda-\left(  2n-t-ib\right)  =2k+t+ib+x+iy-\left(  2n-t-ib\right)
=2(k-n)+2t+x+2ib+iy,
\]
where $b=\alpha+i\beta,$ there exist $m\in\mathbb{Z}$ and $v\in(-1,1)$ such
that%
\[
\lambda-\left(  2n-t-ib\right)  =2m+v+\imath(2\alpha+y).
\]
Here, $\left\vert 2\alpha+y\right\vert \geq\sqrt{1-v^{2}},$ since $\left\vert
\alpha\right\vert \geq1.$ Therefore,
\begin{equation}
\left\vert \lambda-(2n-t-ib)\right\vert ^{2}=\left(  2m+v\right)  ^{2}%
+1-v^{2}=4m^{2}+4mv+1. \tag{33}%
\end{equation}
Now, instead of (31), we use (33), where $s$ and $x\in\lbrack-1,1]$ are
replaced by $m$ and $v\in(-1,1).$ Repeating the proof of (32), we obtain
\begin{equation}
\sum_{n\in\mathbb{Z}}\frac{1}{\left\vert \lambda-(2n-t-ib)\right\vert ^{2}%
}<\frac{1}{4}\pi^{2}. \tag{34}%
\end{equation}
Thus, (29) follows from(32) and (34). Hence, if the point $\lambda$ from the
circle $C(1,2k+t+ib)$ is an eigenvalue of $L_{t}(Q),$ then we arrive at the
contradiction (30). This means that the circle lies in the resolvent set of
$L_{t}(Q).$ In the same way we prove that the circle $C(1,2k-t-ib)$ also lies
in the resolvent set of $L_{t}(Q).$ The lemma is proved.
\end{proof}

The following lemma, which similar to Lemma 1, can be proved in a similar way
for the case $\left\vert \operatorname{Re}b\right\vert <1$.

\begin{lemma}
If $0<\left\vert \operatorname{Re}b\right\vert <1$ and
\begin{equation}
\int\limits_{0}^{\pi}\left\vert a_{1}(x)\pm\left(  a_{2}(x)+b\right)
i\right\vert ^{2}+\left\vert a_{3}(x)-b\pm\imath a_{4}(x)\right\vert ^{2}%
\leq2\pi\left(  \frac{2}{\left(  \operatorname{Re}b\right)  ^{2}}+\frac{1}%
{3}\pi^{2}\right)  ^{-1}, \tag{35}%
\end{equation}
then the circles
\[
C(\left\vert \operatorname{Re}b\right\vert ,2k\pm t\pm ib)=\left\{  \lambda
\in\mathbb{C}:\left\vert \lambda-(2k\pm t\pm ib)\right\vert =\left\vert
\operatorname{Re}b\right\vert \right\}
\]
for all $k\in\mathbb{Z}$ lie in the resolvent set of $L_{t}(Q)$.
\end{lemma}

\begin{proof}
It follows from the argument used in Lemma 1 that it is enough to show that if
$\lambda\in C(\left\vert \operatorname{Re}b\right\vert ,2k\pm t\pm ib),$ then
\begin{equation}
\sum_{n\in\mathbb{Z}}\frac{1}{\left\vert \lambda-(2n+\left(  -1\right)
^{j}(t+ib)\right\vert ^{2}}<\frac{1}{\left(  \operatorname{Re}b\right)  ^{2}%
}+\frac{1}{6}\pi^{2}. \tag{36}%
\end{equation}
for $j=1,2.$ First consider the case $\lambda\in C(\left\vert
\operatorname{Re}b\right\vert ,2k+t+ib)$ and $j=2.$ In this case, we have
$\lambda=2k+t+ib+x+iy$ and $\lambda-(2n+(t+ib)=2s+x+iy$, where
\[
s=k-n,\text{ }x\in\lbrack-\left\vert \operatorname{Re}b\right\vert ,\left\vert
\operatorname{Re}b\right\vert ],\text{ }y=\pm\sqrt{\left(  \operatorname{Re}%
b\right)  ^{2}-x^{2}}.
\]
Therefore, we have
\begin{equation}
\left\vert \lambda-(2n+t+ib)\right\vert ^{2}=\left(  2s+x\right)  ^{2}+\left(
\operatorname{Re}b\right)  ^{2}-x^{2}=4s^{2}+4sx+\left(  \operatorname{Re}%
b\right)  ^{2} \tag{37}%
\end{equation}
and
\[
\sum_{n\in\mathbb{Z}}\frac{1}{\left\vert \lambda-(2n+t+ib)\right\vert ^{2}%
}=\frac{1}{\left(  \operatorname{Re}b\right)  ^{2}}+\sum_{s\in\mathbb{N}}%
\frac{1}{4s^{2}+4sx+\left(  \operatorname{Re}b\right)  ^{2}}+
\]%
\[
\sum_{s\in\mathbb{N}}\frac{1}{4s^{2}-4sx+\left(  \operatorname{Re}b\right)
^{2}}=\frac{1}{\left(  \operatorname{Re}b\right)  ^{2}}+\sum_{s\in\mathbb{N}%
}\frac{2(4s^{2}+\left(  \operatorname{Re}b\right)  ^{2})}{(4s^{2}+\left(
\operatorname{Re}b\right)  ^{2})^{2}-16s^{2}x^{2}}.
\]
Since the function $(4s^{2}+\left(  \operatorname{Re}b\right)  ^{2}%
)^{2}-16s^{2}x^{2}$ on the interval $[-\left\vert \operatorname{Re}%
b\right\vert ,\left\vert \operatorname{Re}b\right\vert ]$ attains its minimum
value at the points $x=\left\vert \operatorname{Re}b\right\vert $ or
$x=-\left\vert \operatorname{Re}b\right\vert ,$ in the both cases, we obtain
\begin{equation}
\sum_{n\in\mathbb{Z}}\frac{1}{\left\vert \lambda-(2n+t+ib)\right\vert ^{2}%
}\leq\frac{1}{\left(  \operatorname{Re}b\right)  ^{2}}+\sum_{s\in\mathbb{N}%
}\frac{1}{\left(  2s+\left\vert \operatorname{Re}b\right\vert \right)  ^{2}%
}+\sum_{s\in\mathbb{N}}\frac{1}{\left(  2s-\left\vert \operatorname{Re}%
b\right\vert \right)  ^{2}}. \tag{38}%
\end{equation}
Now, using the obvious equalities
\[
\sum_{s\in\mathbb{N}}\frac{1}{\left(  2s+\left\vert \operatorname{Re}%
b\right\vert \right)  ^{2}}<\sum_{n=1}^{\infty}\frac{1}{(2s)^{2}}=\frac{1}%
{24}\pi^{2}%
\]
and%
\[
\sum_{s\in\mathbb{N}}\frac{1}{\left(  2s-\left\vert \operatorname{Re}%
b\right\vert \right)  ^{2}}<\sum_{n=1}^{\infty}\frac{1}{(2s-1)^{2}}=\frac
{1}{8}\pi^{2},
\]
we obtain%
\[
\sum_{n\in\mathbb{Z}}\frac{1}{\left\vert \lambda-(2n+(t+ib)\right\vert ^{2}%
}<\frac{1}{\left(  \operatorname{Re}b\right)  ^{2}}+\frac{1}{6}\pi^{2}.
\]
Thus, (36) for the case $j=2$ is proved. In the same way, using the arguments
from the proof of (34), we prove (36) for the case $j=1.$ Now, instead of (27)
and (29) using (35) and (36) and repeating the proof of Lemma 1, we obtain the
proof of this lemma.
\end{proof}

Now, we are ready to prove the main result of this section.

\begin{theorem}
Suppose that either the conditions of Lemma 1 or the conditions of Lemma 2
hold. Then:

$(a)$ All eigenvalues of the operators $L_{t}(Q)$ for all $t\in(-1,1]$ are simple.

$(b)$ $L(Q)$ is a spectral operator.

$(c)$ For each $f$ $\in L_{2}^{m}(-\infty,\infty),$ the following equality
holds:
\[
f=\frac{1}{2}%
{\displaystyle\sum\limits_{j=1,2;\text{ }k\in\mathbb{Z}_{{}}}}
\int\limits_{(-1,1]}a_{k,j}(t)\Psi_{k,j,t}dt,
\]
where the series converges in the $L_{2}^{2}(a,b)$-norm for every
$a,b\in\mathbb{R}.$
\end{theorem}

\begin{proof}
$(a)$ Consider the family of operators $L_{t}(Q_{b}+\varepsilon(Q-Q_{b}))$ for
$\varepsilon\in\lbrack0,1].$ Instead of (26), using
\[
(\lambda(\varepsilon)-(2n+\left(  -1\right)  ^{j}(t+ib))(\Psi_{\lambda
(\varepsilon)},\varphi_{n,j,t})=(\Psi_{\lambda(\varepsilon)},\varepsilon
(Q-Q_{b})\varphi_{n,j,t}),
\]
where $\Psi_{\lambda(\varepsilon)}$ is the normalized eigenfunction of the
operator $L_{t}(Q_{b}+\varepsilon(Q-Q_{b})),$ and repeating the proof of Lemma
1, we obtain that if $\left\vert \operatorname{Re}b\right\vert \geq1$ and (27)
holds, then the circles $C(1,2k\pm t\pm ib),$ for all $k\in\mathbb{Z},$ lie in
the resolvent set of $L_{t}(Q_{b}+\varepsilon(Q-Q_{b}))$ for all
$\varepsilon\in\lbrack0,1].$ Since $L_{t}(Q_{b}+\varepsilon(Q-Q_{b}))$
$L_{t,\varepsilon}$ forms a holomorphic family with respect to $\varepsilon,$
and the operator $L_{t}(Q_{b})$ has exactly one eigenvalue inside each of this
circles, the operator $L_{t}(Q)$ must also have exactly one eigenvalue
(counting multiplicity) in this circle. This implies that all eigenvalues of
$L_{t}(Q)$ are simple. In the same way, we prove that all eigenvalues of
$L_{t}(Q)$ are simple if the conditions of Lemma 2 hold.

$(b)$ Let $\gamma(t)$ be a closed contour lying in the resolvent set
$\rho(L_{t}(Q)).$ Then, there exist pairs $(k_{1},j_{1}),(k_{2},j_{2}%
),\cdot\cdot\cdot,(k_{s},j_{s})$ from $\left\{  (k,j)\in\mathbb{Z}:\left\vert
k\right\vert <N,j=1,2\right\}  $ and a subset $\mathbb{D}(t)$ of the set
\[
\{(n,j):n\in\mathbb{Z},\mid n\mid\geq N,\text{ }j=1,2\}
\]
such that $\gamma(t)$ encloses the eigenvalues $\lambda_{k_{1},j_{1}%
}(t),\lambda_{k_{2},j_{2}}(t),\cdot\cdot\cdot,\lambda_{k_{s},j_{s}}(t)$ and
$\lambda_{n,j}(t)$ for $\left(  n,j\right)  \in\mathbb{D}(t)$, where $N$ is
defined in Theorem 1$(a)$ and does not depend on $t$. Then, we have
\begin{equation}
e(\gamma(t))f=%
{\textstyle\sum\limits_{i=1,2,...,s}}
\frac{1}{\alpha_{k_{i}j_{i}}(t)}(f,\Psi_{k_{i}j_{i},t}^{\ast})\Psi_{k_{i}%
j_{i},t}+%
{\textstyle\sum\limits_{\left(  n,j\right)  \in\mathbb{D}(t)}}
\frac{1}{\alpha_{n,j}(t)}(f,\Psi_{n,j,t}^{\ast})\Psi_{n,j,t}. \tag{39}%
\end{equation}
It is clear that
\begin{equation}
\left\Vert
{\textstyle\sum\limits_{i=1,2,...,s}}
\frac{1}{\alpha_{k_{i}j_{i}}(t)}(f,\Psi_{k_{i}j_{i},t}^{\ast})\Psi_{k_{i}%
j_{i},t}\right\Vert \leq%
{\textstyle\sum\limits_{i=1,2,...,s}}
\frac{\left\Vert f\right\Vert }{\left\vert \alpha_{k_{i}j_{i}}(t)\right\vert
}. \tag{40}%
\end{equation}
Moreover, $\left\vert \alpha_{k_{i}j_{i}}(t)\right\vert $ depends continuously
on $t$ and $\alpha_{k_{i}j_{i}}(t)\neq0$ for all $t\in\lbrack-1,1].$
Therefore, it follows from (21), (24), (39), and (40) that
\[
\left\Vert e(\gamma(t))\right\Vert <m_{6}%
\]
for all $t\in(-1,1]$ and $\gamma(t)\subset\rho\left(  L_{t}\right)  .$ Thus,
the proof of $(b)$ follows from Definition 1.

$(c)$ Since all eigenvalues are simple, the expression $a_{k,j}(t)\Psi
_{k,j,t}$ is integrable for all $j=1,2$ and $k\in\mathbb{Z}.$ Therefore, the
proof of $(c)$ follows from Theorem 3.
\end{proof}

\begin{remark}
Since the eigenvalues of $L_{t}(Q+aI),$ where $a\in\mathbb{C}$ and $I$ \ is
the identity matrix, are obtained by shifting the eigenvalues of $L_{t}(Q),$
and the eigenfunctions remain the same, Theorem 4 continues to hold if (27)
and (35) are replaced, respectively, by
\[
\inf_{a}\int\limits_{0}^{\pi}\left\vert a_{1}(x)+a\pm\left(  a_{2}%
(x)+b\right)  i\right\vert ^{2}+\left\vert a_{3}(x)-b\pm\imath(a_{4}%
(x)+a)\right\vert ^{2}dx\leq\frac{4}{\pi},
\]
and
\[
\inf_{a}\int\limits_{0}^{\pi}\left\vert a_{1}(x)+a\pm\left(  a_{2}%
(x)+b\right)  i\right\vert ^{2}+\left\vert a_{3}(x)-b\pm\imath(a_{4}%
(x)+a)\right\vert ^{2}\leq2\pi\left(  \frac{2}{\left(  \operatorname{Re}%
b\right)  ^{2}}+\frac{1}{3}\pi^{2}\right)  ^{-1}.
\]

\end{remark}

\end{document}